\documentclass[12pt,twoside]{amsart}
\usepackage{amsmath}
\usepackage{amsthm}
\usepackage{amsfonts}
\usepackage{amssymb}
\usepackage{latexsym}
\usepackage[all]{xy}

\date{}
\allowdisplaybreaks[4] \footskip=15pt
\renewcommand{\uppercasenonmath}[1]{}

\numberwithin{equation}{section} \theoremstyle{plain}
\newtheorem*{thm*}{Main Theorem}
\newtheorem{thm}{Theorem}[section]
\newtheorem{cor}[thm]{Corollary}
\newtheorem*{cor*}{Corollary}
\newtheorem{lem}[thm]{Lemma}
\newtheorem*{lem*}{Lemma}

\newtheorem*{que*}{Question}

\newtheorem*{prop*}{Proposition}
\newtheorem{rem}[thm]{Remark}
\newtheorem*{rem*}{Remark}
\newtheorem{exa}[thm]{Example}
\newtheorem*{exa*}{Example}
\newtheorem{df}[thm]{Definition}
\newtheorem*{df*}{Definition}

\newtheorem*{conj*}{Conjecture}

\newtheorem*{ack*}{ACKNOWLEDGEMENTS}

\newcommand{\pf}{\noindent\begin {proof}}
\newcommand{\epf}{\end{proof}}

\begin{document}
\begin{center}
{\large  \bf A note on rings with the summand sum  property}

\vspace{0.8cm} {\small \bf  Liang Shen\\
 \vspace{0.6cm} {\rm
Department of Mathematics, Southeast University
 \\ Nanjing 210096, P.R. China}\\}

 E-mail: {\it lshen@seu.edu.cn}

 \end{center}

%\begin{figure}[b]
%\rule[-2.5truemm]{5cm}{0.1truemm}\\[2mm]
%{\small }
%\end{figure}

\bigskip
\centerline { \bf  ABSTRACT}
 \bigskip
\leftskip10truemm \rightskip10truemm

\noindent A ring $R$ is called right SSP (SIP) if the sum (intersection) of any two direct summands of $R_{R}$ is also a direct summand. Left sides can be defined similarly.  The following are equivalent: (1) $R$ is right SSP. (2) $R$ is right C3 and right SIP. (3) $R$ is left C3 and left SIP. (4) $R$ is left SSP.  It is also shown that  (1) $R$ is a von-Neumann regular ring if and only if $\mathbb{M}_{2}(R)$ is right SSP if and only if $\mathbb{M}_{n}(R)$ is right SSP for some $n>1$; (2) $R$ is a semisimple  ring if and only if the column finite  matrix ring $\mathbb{C}\mathbb{F}\mathbb{M}_{\Lambda}(R)$ is right SSP for a countably infinite set $\Lambda$ if and only if the column finite  matrix ring $\mathbb{C}\mathbb{F}\mathbb{M}_{\Lambda}(R)$ is right SSP for any infinite set $\Lambda$. Some known results are improved. \leftskip0truemm \rightskip0truemm
 \\

\bigskip
\section { \bf INTRODUCTION}
\bigskip

\indent Throughout this paper, rings are associative with
identity and modules are unitary modules.  We denote by $\mathbb{N}$  the set of natural numbers. For a ring $R$, $\mathbb{M}_{n}(R)$ denotes the ring of all $n\times n$ matrices over $R$. Let $\Lambda$ be an infinite set. $\mathbb{C}\mathbb{F}\mathbb{M}_{\Lambda}(R)$ means the column finite $card(\Lambda)\times card(\Lambda)$ matrix ring over $R$, where $card(\Lambda)$ is the cardinality of $\Lambda$.   For a module $M$, $M^{(A)}$ is the direct sum of copies of $M$ indexed by a set $A$.  We use $N\leq_{\oplus}M$ to show that  $N$ is a direct summand of $M$.  And use End($M$) to denote the ring of endomorphisms of $M$.  For a subset $X$ of a ring $R$, the left annihilator of $X$ in $R$ is ${\bf l}(X)=\{r\in R: rx=0$ for all $x\in X\}$. Right annihilators are defined
analogously. \\
\indent Let $R$ be a ring. Recall that an  $R$-module $M$ has {\it the summand intersection property}  (SIP) if the intersection of any two direct summands of  $M$ is also a direct summand of $M$.  $M$ has {\it the summand sum property}  (SSP) if the sum of any two direct summands of  $M$ is also a direct summand of $M$. The two dual definitions of modules have been studied by several authors such as Wilson \cite{W86}, Garcia \cite{G89}, Hausen \cite{H89}, Alkan,  Harmanci \cite{AH02} and so on. \\
 \indent    It is well-known that the above two dual definitions of modules can't inform each other. But when we especially concentrate on such properties of rings, several unexpected and interesting results appear. A ring $R$ is called \emph{ right SSP (SIP)} if the sum (intersection) of any two direct summands of the right $R$-module $R_{R}$ is also a direct summand. Left sides can be defined similarly.  Recall that an $R$-module $M$ is called a \emph{C3} module if whenever $N\leq_{\oplus}M$ and $K\leq_{\oplus}M$  such that $N\cap K=0$, then $N+K\leq_{\oplus}M$. A ring $R$ is called \emph{right (left) C3} if $R_{R}$ ($_{R}R$) is a C3 module. It is clear that a right SSP ring is right C3. By Theorem 2.4 and Theorem 2.7,  a ring $R$ is right SSP if and only if it is right C3 and right SIP if and only if it is left C3 and left SIP if and only if it is left SSP.  Several examples are given. It is shown in Theorem 2.14 that a ring $R$ is von-Neumann regular if and only if $\mathbb{M}_{2}(R)$ is  SSP if and only if  $\mathbb{M}_{n}(R)$ is  SSP for some $n> 1$ if and only if $\mathbb{M}_{n}(R)$ is  SSP for any $n> 1$.  At last, in Theorem 2.17,  it is proved that  a ring $R$ is  semisimple  if and only if  $\mathbb{C}\mathbb{F}\mathbb{M}_{\Lambda}(R)$ is right SSP for a countably infinite set $\Lambda$ if and only if  $\mathbb{C}\mathbb{F}\mathbb{M}_{\Lambda}(R)$ is right SSP for any infinite set $\Lambda$. Some known results are improved.

\bigskip
\section { \bf  RESULTS}
\bigskip
\begin{df}\label{def 2.1}
 {\rm A ring $R$ is called right SSP (SIP) if the sum (intersection) of any two direct summands of the right $R$-module $R_{R}$ is also a direct summand. Left SSP (SIP) rings can be defined similarly. A ring $R$ is called SSP (SIP) if $R$ is right and left SSP (SIP).}
\end{df}
First we show  that the SSP definitions of rings are left-right symmetric.
\begin{lem}\label{lem 2.2 }
 Let e and f be two idempotents of a ring $R$.
Then
\begin{center}
eR+fR $\leq_{\oplus}R_{R}$ if and only if (1-e)fR
$\leq_{\oplus}R_{R}$.
\end{center}
\end{lem}
\begin{proof} It is easy to prove that $eR+fR=eR\oplus(1-e)fR$. If $eR+fR
\leq_{\oplus}R_{R}$, then there exists a right ideal $T$ of $R$
such that $(eR+fR)\oplus T=R_{R}$. So $eR\oplus(1-e)fR\oplus
T=R_{R}$. This implies that  $(1-e)fR\leq_{\oplus}R_{R}$.
Conversely, assume $(1-e)fR\leq_{\oplus}R_{R}$. Since
$(1-e)fR\leq (1-e)R$, $(1-e)fR$ is also a direct summand of
$(1-e)R$. So there is a right ideal $K$ of $R$ such that
$(1-e)fR\oplus K=(1-e)R$. Thus
$R_{R}=eR\oplus(1-e)R=eR\oplus(1-e)fR\oplus K=(eR+fR)\oplus K$.
So $eR+fR \leq_{\oplus}R_{R}$.
\end{proof}

Recall that an element $a$ of $R$ is called {\it regular} if there
exists $b\in R$ such that $a=aba$. $R$ is called (\emph{von-Neumann}) \emph{regular} if every element of $R$ is regular. The following lemma is
well-known.

\begin{lem}\label{2.3}
Let a be an element of a ring R.
The following are equivalent.
\begin{enumerate}
\item $a$ is regular.
\item $aR\leq_{\oplus} R_{R}$.
\item $Ra\leq_{\oplus}$$ _{R}R$.
\end{enumerate}
\end{lem}

The next result can also be derived from \cite[Proposition 2.2]{G89}. To be self-contained, we give the
direct proof through idempotents of rings.

\begin{thm}\label{thm 2.4}
 The following are equivalent for a ring R.
\begin{enumerate}
\item R is right SSP.
\item For any two idempotents e and f of R,
$efR\leq_{\oplus}R_{R}$.
\item R is left SSP.
\item For any two idempotents e and f of R,
$Ref\leq_{\oplus}$$_{R}R$.
\end{enumerate}
\end{thm}
\begin{proof}
By Lemma 2.2, it is easy to show that (1)$\Leftrightarrow$(2)
and (3)$\Leftrightarrow$(4). According to Lemma 2.3,  (2)$\Leftrightarrow$(4).
\end{proof}

By the above theorem, we have

 \begin{exa}\label{exa 2.5}
 {\rm Regular rings, abelian rings
 (especially commutative rings) are SSP. Recall that a ring $R$ is called abelian if its
  idempotents are contained in the center of $R$.}
  \end{exa}

  Since a ring $R$ is right SSP if and only if it is left SSP,  we will only consider SSP rings.

  \begin{lem}\label{lem 2.6}
  \cite[Lemma 19 (1)]{AH02}
  Let M be a C3 module. If M has the SIP, then M has the SSP.
  \end{lem}
\begin{thm}\label{thm 2.7}
The following are equivalent for a ring $R$.
\begin{enumerate}
\item $R$ is SSP.
\item $R$ is right C3 and right SIP.
\item $R$ is left C3 and left SIP.
\end{enumerate}
\end{thm}
\begin{proof} We only prove (1)$\Leftrightarrow$(2). (1)$\Leftrightarrow$(3) is similar.\\
(2)$\Rightarrow$(1) is obtained by Theorem 2.4 and Lemma 2.6.\\
For (1)$\Rightarrow$(2), since $R$ is SSP, $R$ is right C3. Now we show that $R$ is right SIP. Let $e$ and $f$ be two idempotents of $R$. Then
\begin{center}
$eR\cap fR={\bf r}(1-e)\cap{\bf r}(1-f)={\bf r}(R(1-e)+R(1-f))$.

\end{center}As $R$ is SSP, there exists an idempotent $g$ of $R$ such that
$R(1-e)+R(1-f)=Rg$. Thus $(eR\cap fR)={\bf
r}(g)=(1-g)R\leq_{\oplus}R_{R}$. So $R$ is right SIP.
\end{proof}

According to the above theorem, we give the following three  remarks.
\begin{rem}\label{rem 2.8}
{\rm A left SIP ring may not be right SIP.}
\end{rem}
\begin{proof}
By \cite[Corollary 2.6 (iii)]{G89}, a ring $R$ is left (right) semihereditary if and only if every ring which is Morita-equivalent to $R$ is left (right) SIP.
Since a left semihereditary ring may not be right semihereditary, a left SIP ring may not be right SIP.
\end{proof}

\begin{rem}\label{rem 2.9}
{\rm  A right SIP ring may not be SSP.}
\end{rem}
\begin{proof}
Let $R=\left(\begin{array}{cc}
K&K\\
0&K
\end{array}\right)$ be the ring of upper triangular matrices over a field $K$, $N=\left(\begin{array}{cc}
0&K\\
0&K
\end{array}\right)$, $L=\left(\begin{array}{cc}
K&K\\
0&0
\end{array}\right)$ left ideals of $R$ and let $M=R/L$. Set $U=N\oplus M$ and $S=End(_{R}U)$. By \cite[Remark, Page 81]{G89}, $S$ is two-sided SIP
but  not SSP
\end {proof}

\begin{rem}\label{rem 2.10}
A right C3 and left SIP ring may not be SSP.
\end{rem}
\begin{proof}
 Set $R=\left(\begin{array}{ccc}
\mathbb{F}_{2}&0&\mathbb{F}_{2}\\
0&\mathbb{F}_{2}&0\\
0&0&\mathbb{F}_{2}
\end{array}\right)$, where $\mathbb{F}_{2}$ is the field with only two elements 0 and 1. It is clear that $R$ is a ring. We show that $R$ is a right C3 and left SIP ring but not SSP. Let $e_{ij}$ denote the matrix units in $R$. The following are all nontrivial idempotents of $R$: $E_{1}=e_{11}$, $E_{2}=e_{22}$, $E_{3}=e_{33}$, $E_{4}=e_{11}+e_{22}$, $E_{5}=e_{11}+e_{33}$, $E_{6}=e_{22}+e_{33}$, $E_{7}=e_{11}+e_{13}$, $E_{8}=e_{33}+e_{13}$, $E_{9}=e_{11}+e_{22}+e_{13}$, $E_{10}=e_{22}+e_{33}+e_{13}$. Thus, all nontrivial right (left) direct summands of $R$ are $E_{n}R$ ($RE_{n}$), $n=1, \ldots, 10$. Through direct computations, $R$ is right C3 and left SIP. But $R$ is not SSP because $RE_{4}+RE_{7}$ is not a left direct summand of $R$.
\end{proof}

\begin{thm}\label{thm 2.11}
Let e be an idempotent of a ring R. If R is SSP, then eRe is also SSP.
\end{thm}
\begin{proof} Suppose $f$ and $g$ are two idempotents of $S=eRe$. We show that $fS+gS\leq_{\oplus}S_{S}$. It is clear that  $f$ and
$g$ are  idempotents of $R$. Since $R$ is SSP, there exists
an idempotent $h$ of $R$ such that $fR+gR=hR$. As $f$, $g\in
S$, $h=eh$. Hence $(ehe)^{2}=eheehe=eh^{2}e=ehe\in S$. This
shows that $ehe$ is an idempotent of $S$. Then
\begin{center}
$fS+gS=feRe+geRe=fRe+gRe=(fR+gR)e=(hR)e=ehRe$.
\end{center} Since $ehRe=hRe=h^{2}Re=eheehRe\subseteq ehe(eRe)\subseteq ehRe$,
$fS+gS=eheS$. Thus, $S$ is an SSP ring.
 \end{proof}

It is easy to see that an SSP ring may not be regular. But when the rings are $n\times n$ matrix ring in which $n>1$, the two definitions are equivalent by the following Theorem 2.14.
 \begin{lem}\label{lem 2.12}
 \cite[Corollary 2.4 (i)]{G89}
Let M be a right $R$-module, S=End($M_{R}$). If M is quasi-projective, then M has the SSP if and only if
S is SSP.
\end{lem}

\begin{lem}\label{lem 2.13}
\cite[Proposition 2.8]{G89}
Let M be a quasi-projective right R-module, S=End($M_{R}$) is regular if and only if $M^{2}$ has the SSP.
\end{lem}

\begin{thm}\label{thm 2.14}
 The following are equivalent for a ring $R$.
 \begin{enumerate}
 \item R is  regular.
 \item $\mathbb{M}_{2}(R)$ is  SSP.
 \item $\mathbb{M}_{n}(R)$ is  SSP for some $n> 1$.
 \item $\mathbb{M}_{n}(R)$ is  SSP for any $n> 1$
 \end{enumerate}

\end{thm}
\begin{proof}
Since a matrix ring over a regular ring is also regular, (1)$\Rightarrow$(2) and (1)$\Rightarrow$(4)$\Rightarrow$(3). It is obvious that End($R_{R}^{n}$)$\cong\mathbb{M}_{n}(R), \forall n\geq 1$. By the above two lemmas, (2)$\Rightarrow$(1). Because a direct summand of a module with the SSP also has the SSP (see \cite[Proposition 1.2]{G89}), (3)$\Rightarrow$(2) is obtained by Lemma 2.12.
\end{proof}

 The following  example shows that SSP is not a Morita invariant.

\begin{exa} \label{exa 2.15}
{\rm The ring of integers $\mathbb{Z}$ is an SSP ring, but
M$_{2\times2}$($\mathbb{Z}$) is not SSP.}
\end{exa}
\begin{proof}
Since $\mathbb{Z}$ is commutative, by Example 2.5, $\mathbb{Z}$ is SSP. But M$_{2\times2}$($\mathbb{Z}$) is not SSP. If it is SSP, by the above theorem,  $\mathbb{Z}$ is regular. This is a contradiction.
\end{proof}

 At last, we give a new connection between semisimple rings and SSP rings. The next lemma was obtained by Yiqiang Zhou. To be completeness, we show the proof. Recall that an $R$-module $M$ is called a \emph{C2} module if every submodule that is isomorphic to a direct summand of $M$ is itself a direct summand of $M$.
\begin{lem}\label{lem 2.16}
{\rm (Zhou)}
Let R be a ring and  M a right $R$-module. If the direct sum $M\oplus M$ is a C3 module, then $M$ is a C2 module.
\end{lem}
\begin{proof}
Assume $K$ is a submodule of $M$ that is isomorphic to a direct summand $L$ of $M$.  We want to show that $K$ is also a direct summand of $M$. Let $f$ be the isomorphism from $K$ to $L$. Set $K'=\{(x, f(x)): x\in K\}$,  $L'=0\oplus L$ and $M'=M\oplus 0$. Then $K'+M'=M\oplus L$ is a direct summand of $M\oplus M$. Since $K'\cap M'=0$, $K'$ is also a direct summand of $M\oplus M$. It is clear that  $K'\cap L'=0$ and $L'$ is a direct summand of $M\oplus M$. Since $M\oplus M$ is a C3 module. $K'+L'=K\oplus L$ is a direct summand of $M\oplus M$. As $K\oplus 0$ is a direct summand of $K\oplus L$, $K\oplus 0$ is also a direct summand of $M\oplus M$. This shows that $K\oplus 0$ is a direct summand of $M\oplus 0$. Now it is clear that $K$ is a direct summand of $M$.
\end{proof}
\begin{thm}\label{thm 2.17}
The following are equivalent for a ring R.
\begin{enumerate}
\item R is  semisimple.
\item $\mathbb{C}\mathbb{F}\mathbb{M}_{\Lambda}(R)$ is SSP for a countably infinite set $\Lambda$.
\item $\mathbb{C}\mathbb{F}\mathbb{M}_{\Lambda}(R)$ is  SSP for some infinite set $\Lambda$.
\item $\mathbb{C}\mathbb{F}\mathbb{M}_{\Lambda}(R)$ is  SSP for any infinite set $\Lambda$.
\end{enumerate}
\end{thm}
\begin{proof}
If $R$ is semisimple, it is obvious that every $R$-module has the SSP. Let  $\Lambda$ be an infinite set. It is not difficult to prove  End($R_{R}^{(\Lambda)})\cong \mathbb{C}\mathbb{F}\mathbb{R}_{\Lambda}(R)$. By Lemma 2.12,  (1)$\Rightarrow$(4)$\Rightarrow$(3).
It is clear that  a countable direct sum  copies of $R_{R}$ can be looked as a direct summand of any infinite direct sum copies of $R_{R}$. Then (3)$\Rightarrow$(2)
 can be obtained by Lemma 2.12.\\
  At last, we prove (2)$\Rightarrow$(1). If $\mathbb{C}\mathbb{F}\mathbb{M}_{\Lambda}(R)$ is SSP for a countably infinite set $\Lambda$, by Lemma 2.12, $R_{R}^{(\mathbb{N})}$ has the SSP. Since the SSP is inherited by direct summands and $R_{R}^{2}$ can be looked as a direct summand of $R_{R}^{(\mathbb{N})}$,  $R_{R}^{2}$  has the SSP. Then $\mathbb{M}_{2}(R)$ is a regular ring by Lemma 2.13. This implies $R$ is a regular ring. Now we show that $R$ is a right perfect ring. First we show that $R_{R}^{(\mathbb{N})}$ is a C2 module. Since $R_{R}^{(\mathbb{N})}$ has the SSP,  $R_{R}^{(\mathbb{N})}$ is a C3 module. As  $(R^{(\mathbb{N})}_{R}\oplus R^{(\mathbb{N})}_{R})\cong R^{(\mathbb{N})}_{R}$, $R^{(\mathbb{N})}_{R}\oplus R^{(\mathbb{N})}_{R}$ is also a C3 module. According to Lemma 2.16, $R^{(\mathbb{N})}_{R}$ is a C2 module. In order to prove that $R$ is a right perfect ring, by \cite[Theorem 28.4]{AF92}, we only need to show that $R$ satisfies $DCC$ on principal left ideals of $R$. The following method is owing to Bass. Let $Ra_{1}\supseteq Ra_{2}a_{1}\supseteq \cdots$ be any descending chain of principal left ideals of $R$. Set $F=R^{(\mathbb{N})}_{R}$ with free basis $x_{1}, x_{2},\ldots$ and $G$ be the submodule of $F$ spanned by $y_{i}=x_{i}-x_{i+1}a_{i}, i\in\mathbb{N}$. By \cite[Lemma 28.1]{AF92}, $G$ is free with basis $y_{1}, y_{2},\ldots$. Thus $G\cong F$. $F$ is a C2 module implies that $G$ is a direct summand of $F$. Then by \cite[Lemma 28.2]{AF92}, the chain $Ra_{1}\supseteq Ra_{2}a_{1}\supseteq \cdots$ terminates. Since $R$ is both regular and right perfect, $R$ is a semisimple ring.
\end{proof}

 By Theorem 2.17 and Lemma 2.12, we have

\begin{cor}\label{cor 2.18}
R is a semisimple ring if and only if $R_{R}^{(\mathbb{N})}$ has the SSP.
\end{cor}

\begin{cor}\label{cor 2.19}
\cite[Theorem 9]{AH02}
Let $R$ be a ring. The following are equivalent.
\begin{enumerate}
\item R is semisimple.
\item Every $R$-module has the SSP.
\item Every projective module has the SSP
\end{enumerate}
\end{cor}
\begin{proof}
(1)$\Rightarrow$(2)$\Rightarrow$(3) is obvious. By the above corollary, (3)$\Rightarrow$(1).
\end{proof}

 \centerline {\bf ACKNOWLEDGMENTS}

\end{document}